\documentclass[a4paper]{article}
\usepackage{fullpage}
\usepackage{tikz}

\usepackage[pdftex,breaklinks,colorlinks,
linkcolor=blue,
citecolor=blue,
urlcolor=blue]{hyperref}
\newcommand{\pr}[1]{{\color{red}#1}}

\usepackage{enumitem}
\usepackage{amsmath,amssymb,amsfonts,amsthm}

\newtheorem{theorem}{Theorem}
\newtheorem{conjecture}[theorem]{Conjecture}
\newtheorem{lemma}[theorem]{Lemma}
\newtheorem{proposition}[theorem]{Proposition}

\theoremstyle{definition}
\newtheorem{definition}[theorem]{Definition}
\newtheorem{remark}[theorem]{Remark}

\author{Pablo Romero\footnote{Facultad de Ingenier\'ia, Universidad de la Rep\'ublica, Montevideo, Uruguay. E-mail address: \texttt{promero@fing.edu.uy}}}

\date{}

\makeatletter%
\begin{document}

\title{There are finitely many\\ uniformly most reliable graphs of corank 5}

\maketitle

\begin{abstract}\let\thefootnote\relax
If $G$ is a simple graph and $\rho\in[0,1]$, the reliability $R_G(\rho)$ is the probability of $G$ being connected after each of its edges is removed independently with probability $\rho$. A simple graph $G$ is a \emph{uniformly most reliable graph} (UMRG) if $R_G(\rho)\geq R_H(\rho)$ for every $\rho\in[0,1]$ and every simple graph $H$ on the same number of vertices and edges as $G$. Boesch conjectured that, if $n$ and $m$ are such that there exists a connected simple graph on $n$ vertices and $m$ edges, then there also exists a UMRG on the same number of vertices and edges. Some counterexamples to Boesch's conjecture appeared in the literature. It is known that Boesch's conjecture holds whenever the corank, defined as $c=m-n+1$, is at most $4$ (and the corresponding UMRGs are fully characterized). Ath and Sobel conjectured that Boesch's conjecture holds whenever the corank $c$ is between $5$ and $8$, provided that the number of vertices is at least $2c-2$. It is known that for each positive integer $s$ there is no uniformly most reliable graph of corank $5$ and $12s+4$ vertices. Here it is proved that there are only finitely many uniformly most reliable graphs of corank $5$. This is in strong contrast with classes of graphs whose corank is not greater than $4$ for which uniformly most reliable graphs always exist.
\end{abstract}

\renewcommand{\labelitemi}{--}

\section{Introduction}\label{section:intro}
If $G$ is a simple graph and $\rho \in [0,1]$, the  \emph{reliability of $G$ with failure probability $\rho$}, denoted $R_G(\rho)$, is the probability of $G$ being connected after each of its edges is removed independently with probability $\rho$. Given positive integers $n$ and $m$ such that $n-1\leq m \leq \binom{n}{2}$, the question arises as to whether there exists a graph $G$ in the class $\mathcal{C}_{n,m}$ of connected simple graphs on $n$ vertices and $m$ edges such that 
$R_G(\rho) \geq R_{H}(\rho)$ for every $H$ in $\mathcal{C}_{n,m}$ and every $\rho\in[0,1]$. Such a simple graph $G$ is called a \emph{uniformly most reliable graph} (UMRG). This concept was introduced in 1986 by Boesch in his seminal article~\cite{1986-Boesch}. 


The corank of each graph $G$ in $\mathcal{C}_{n,m}$ is defined as $m-n+1$. 
We will also say that the corank of the class $\mathcal{C}_{n,m}$ is $m-n+1$. Regarding dense graphs, it is known that there exists infinitely many classes $\mathcal{C}_{n,m}$ in which there is no UMRG thus Boesch's conjecture was disproved~\cite{2014-Brown,1981-Kelmans,1991-Myrvold}. Regarding sparse graphs, in each nonempty class $\mathcal{C}_{n,m}$ having corank at most $4$  there exists precisely one UMRG. The reader is invited to consult the survey~\cite{2022-Survey} for further details. 

Ath and Sobel~\cite{2000-Ath} proposed the following weaker conjecture in 2000.
\begin{conjecture}[Ath-Sobel~\cite{2000-Ath}]\label{conjecture:ath-sobel}
If a nonempty class $\mathcal C_{n,m}$ has corank $c\in\{5,6,7,8\}$ and $n\geq 2c-2$, then $\mathcal C_{n,m}$ contains at least one UMRG.
\end{conjecture}
Furthermore, explicit candidates for such UMRGs were also proposed in~\cite{2000-Ath}. 
Conjecture~\ref{conjecture:ath-sobel} was disproved when $c=5$. 
\begin{theorem}[\cite{2023-RomeroSafe}]\label{theorem:principal}
For each positive integer $s$ there is no UMRG in the class 
$\mathcal{C}_{12s+4,12s+8}$. 
\end{theorem}

It is known~\cite{1991-Myrvold} that there is precisely one UMRG 
in each of the classes of corank $5$ and $n$ vertices where $n\in \{5,6,7,8\}$. By Theorem~\ref{theorem:principal}, there exists infinitely many classes of corank $5$ with no UMRG. Here, we will prove that there are only finitely many classes $\mathcal{C}_{n,m}$ of corank $5$ having some UMRG. As each class $\mathcal{C}_{n,m}$ is finite, there are only finitely many UMRG having corank $5$. This is in strong contrast with classes of graphs whose corank is not greater than $4$ for which UMRG always exist. 

This article is organized as follows. A proof strategy is given in Section~\ref{section:strategy}. A background including some concepts from graph theory as well as preliminary results is presented in Section~\ref{section:basics}. 
The concept of vertex-fair graph is introduced in Section~\ref{section:nodefair}. The main result of this work is given in Section~\ref{section:main}.


\section{Proof strategy}\label{section:strategy}
Let $G$ be a graph in $\mathcal{C}_{n,m}$. An \emph{edge-cut} of $G$ is a set $F$ of edges of $G$ such that $G-F$ is disconnected. A \emph{$k$-edge-cut} is an edge-cut of size $k$. 
We denote $\mu_k(G)$ the number of $k$-edge-cuts in $G$. Clearly, for each $\rho\in [0,1]$,
\begin{equation*}
R_G(\rho) = 1-\sum_{k=0}^{m}\mu_k(G)\rho^k(1-\rho)^{m-k}.    
\end{equation*}

The following theorem can be proved using elementary calculus.
\begin{theorem}[Brown and Cox~\cite{2014-Brown}]\label{theorem:optlocal}
Let $G$ and $H$ be graphs in $\mathcal{C}_{n,m}$.
\begin{enumerate}[label=(\roman*)]
    \item\label{it1:optlocal} If there exists $i\in \{0,1,\ldots,m\}$ such that $\mu_{k}(G)=\mu_k(H)$ for all $k\in\{0,1,\ldots,i-1\}$ and $\mu_i(G)<\mu_i(H)$, then there exists some $\delta>0$ such that $R_G(\rho)>R_H(\rho)$ for every $\rho \in (0,\delta)$.
    \item\label{it2:optlocal} If there exists $j\in \{0,1,\ldots,m\}$ such that $\mu_k(G)=\mu_k(H)$ for all $k\in\{j+1,j+2,\ldots,m\}$ and $\mu_j(G)<\mu_j(H)$, then 
    there exists some $\delta>0$ such that $R_G(\rho)>R_H(\rho)$ for every $\rho \in (1-\delta,1)$.
\end{enumerate}
\end{theorem}
Consider any nonempty class $\mathcal{C}_{n,m}$ and any graph $G$ in $\mathcal{C}_{n,m}$. 
If $k>m-n+1$ then, clearly, $\mu_k(G)=\binom{m}{k}$. Let $t(G)$ be the number of spanning trees of $G$. Observe that $\mu_{m-n+1}(G)=\binom{m}{m-n+1}-t(G)$.  
\begin{definition}
A graph $H$ in $\mathcal{C}_{n,m}$ is \emph{$t$-optimal} if $\mu_{m-n+1}(H)\leq \mu_{m-n+1}(G)$ for each $G$ in $\mathcal{C}_{n,m}$.    
\end{definition}
The following remark is essential for our proof strategy of nonexistence. It is just an application of Theorem~\ref{theorem:optlocal}\ref{it2:optlocal} using $j=m-n+1$.
\begin{remark}\label{remark:toptimality}
Each uniformly most reliable graph in $\mathcal{C}_{n,m}$ is $t$-optimal.    
\end{remark}

Now, consider any nonempty class $\mathcal{C}_{n,m}$. Sort all graphs $G$ in $\mathcal{C}_{n,m}$ using the lexicographic order among all tuples $(\mu_0(G),\mu_1(G),\ldots,\mu_m(G))$. Define $\mathcal{C}_{n,m}^{0}$ as $\mathcal{C}_{n,m}$ and $\mathcal{C}_{n,m}^{i+1}=\{G:G\in \mathcal{C}_{n,m}^{i}, \, \forall H\in \mathcal{C}_{n,m}^{i}, \,\, \mu_{i+1}(G)\leq \mu_{i+1}(H)\}$ for each $i\in \{0,1,\ldots,m-1\}$. As $\mathcal{C}_{n,m}$ is finite and nonempty, each set  $\mathcal{C}_{n,m}^{i}$ is finite and nonempty as well. Our proof strategy is  summarized in the following lemma.
\begin{lemma}\label{lemma:strategy}
If $\mathcal{C}_{n,m}$ is nonempty, $\mathcal{C}_{n,m}^{j}=\{G\}$ for some positive integer $j$ and $G$ is not $t$-optimal, then there is no UMRG in $\mathcal{C}_{n,m}$.
\end{lemma}
\begin{proof}
As $G$ is not $t$-optimal, by Remark~\ref{remark:toptimality} we know that $G$ is not UMRG. 
Let $H$ be any other graph in $\mathcal{C}_{n,m}$. 
As $\mathcal{C}_{n,m}^{j}=\{G\}$, there exists some integer $i\in \{0,1,\ldots,j-1\}$ such that $\mu_k(G)=\mu_k(H)$ for all $k\in\{0,1,\ldots,i-1\}$ but $\mu_i(G)<\mu_i(H)$. By Theorem~\ref{theorem:optlocal}\ref{it1:optlocal}, there exists some $\delta>0$ such that $R_{G}(\rho)>R_{H}(\rho)$ for every $\rho\in (0,\delta)$. Consequently, $H$ is not UMRG. 
The lemma follows. \qed
\end{proof}

In Section~\ref{section:main} we will find the only graph $G_n$ in $\mathcal{C}_{n,n+4}$ such that $\mathcal{C}_{n,n+4}^{4}=\{G_n\}$. Then we will 
find a graph $H_n$ in each set $\mathcal{C}_{n,n+4}$ such that $\mu_{5}(H_n)<\mu_5(G_n)$ for all $n\geq n_0$. The main theorem will then follow from Lemma~\ref{lemma:strategy}. 

\section{Background}\label{section:basics}
In this section we include some concepts that already appeared in~\cite{2023-RomeroSafe} as well as some preliminary results that will be useful for our purpose.

All graphs in this work are finite and undirected. We denote the path, the cycle, and the complete graph on $n$ vertices by $P_n$, $C_n$, and $K_n$, respectively. The Wagner graph $W$ and the cube graph $Q$ are depicted in Figure~\ref{figure:cubicosB}. For each positive integer $p$ such that $p\geq 2$, the M\"obius graph, denoted $M_p$, consists of $C_{2p}$ plus $p$ edges each one joining opposite vertices in $C_{2p}$. 

The \emph{edge-connectivity} of a graph $G\neq K_1$, denoted  $\lambda{(G)}$, is the minimum $k$ such that $G$ has a $k$-edge-cut. Let $F$ be an edge-cut of $G$. We say $F$ \emph{separates} a set $S$ of vertices of $G$ if $F$ contains all edges with precisely one endpoint in $S$ and no edge with both endpoints in $S$. Notice that $F$ may also contain some edges with no endpoints in $S$. If $S=\{v\}$ for some vertex $v$ of $G$, we say that $F$ is \emph{vertex-separating}. If $S=\{u,v\}$ where $u$ and $v$ are the endpoints of an edge $e$ of $G$, we say that $F$ is \emph{edge-separating}. An edge-cut is \emph{nontrivial} if it is neither vertex-separating nor edge-separating. If $S$ induces a graph $H$ in $G$, we say $F$ is \emph{$H$-separating}.

Let $G$ be a graph with no loops. An edge $e$ is \emph{incident} to a vertex $v$ if $v$ is an endpoint of $e$. The \emph{degree} $d_G(v)$ of a vertex $v$ of $G$ is the number of edges incident to it. We say $G$ is \emph{cubic} if all its vertices have degree $3$. Two edges are \emph{nonincident}, \emph{incident}, or \emph{parallel} if they share precisely 0, 1, or 2 endpoints, respectively. By \emph{subdividing $k$ times} an edge with endpoints $x$ and $y$, we mean replacing the edge $xy$ by $k+1$ edges $xv_1,v_1v_2,\ldots,v_{k-1}v_k, v_ky$, where $v_1,v_2,\ldots,v_k$ are $k$ new vertices of degree $2$ each.

A simple graph $G$ is \emph{2-connected} if it has at least $3$ vertices, it is connected, and $G-v$ is connected for all $v$ in $V(G)$. Let $G$ be a $2$-connected simple graph having more edges than vertices. A \emph{chain $\gamma$ of $G$} is the edge set of a path $P$ in $G$, where all internal vertices of $P$ (if any) have degree $2$ in $G$ and $P$ has two distinct endpoints of degree greater than $2$ in $G$ each. The \emph{endpoints} of $\gamma$ are those of $P$ and $\gamma$ is \emph{incident} to a vertex $v$ if $v$ is one of its endpoints. The \emph{internal vertices} of $\gamma$ are those of $P$. By \emph{removing} $\gamma$ from $G$, we mean removing the edges and internal vertices of $\gamma$ (but not its endpoints). The graph that results by removing $\gamma$ from $G$ is denoted $G\ominus\gamma$. If $\mathcal{H}$ is a set of chains of $G$, we denote $G\ominus \mathcal{H}$ the graph that arises from $G$ by removing all the chains in $\mathcal{H}$. By \emph{collapsing $\gamma$} we mean removing $\gamma$ and adding an edge with the same endpoints as $\gamma$. We denote $\Gamma(G)$ the set of all chains of $G$. The \emph{distillation of $G$}, denoted $D(G)$, is the graph that arises from $G$ by collapsing all of its chains. Clearly, $G$ arises from $D(G)$ by a sequence of (possibly zero) subdivisions. 

Let $G$ be in $\mathcal{C}_{n,m}$. If $k\in\{0,1,\ldots,m\}$, we say $G$ is \emph{min-$\mu_k$} if $\mu_k(G) \leq \mu_k(H)$ for every $H$ in $\mathcal{C}_{n,m}$. Theorem~\ref{theorem:strong} restricts our study to $2$-connected graphs. 
\begin{theorem}[Wang \cite{1994-Wang}]\label{theorem:strong}
Let $G$ be a simple graph on $n$ vertices and $m$ edges such that $m>n$. If $G$ is min-$\mu_k$ for some $k\in \{\lambda(G),\lambda(G)+1,\ldots,m-n+1\}$, then $G$ is 2-connected.
\end{theorem}


If $c$ is a positive integer, $[c]$ denotes the set $\{1,2,\ldots,c\}$. If $k$ is a nonnegative integer, then the family of all the subsets of a set $S$ with cardinality $k$ is denoted $\binom{S}{k}$. Let $G$ be a $2$-connected simple graph having more edges than vertices. Let $\Gamma(G)$ be the set of all chains of $G$. Moreover, if $k$ is a nonnegative integer, let $\Gamma^{(k)}(G)$ be the family of all subsets of $\Gamma(G)$ of size $k$; i.e., $\Gamma^{(k)}(G)=\binom{\Gamma(G)}k$. We also let
\[ \Gamma^{(k)}_-(G)=\{\mathcal{H}\in\Gamma^{(k)}(G):\,G\ominus \mathcal{H}\text{ is disconnected}\}. \]

The following lemma gives a closed form for the number of $k$-edge-cuts of any $2$-connected simple graph having more edges than vertices.

\begin{lemma}[\cite{2023-RomeroSafe}]\label{lemma:objetivo} 
For each $2$-connected graph $G$ in $\mathcal{C}_{n,m}$ such that $m> n$ and each $k\in \{0,1,\ldots,m\}$,
\begin{equation}\label{eq:mk}
  \mu_k(G)=\binom mk-\sum_{\mathcal{H}\in\Gamma^{(k)}(G)}\prod_{\gamma\in \mathcal{H}}\ell(\gamma)+\sum_{\mathcal{H}\in\Gamma^{(k)}_-(G)}\prod_{\gamma\in \mathcal{H}}\ell(\gamma).
\end{equation}
\end{lemma}

If the graph $G$ in Lemma~\ref{lemma:objetivo} has precisely $t$ chains and the lengths of all its chains are $\ell_1,\ell_2,\ldots,\ell_t$, then the second term of the right-hand side of~\eqref{eq:mk} is
\begin{equation}\label{eq:aux} \sum_{\mathcal{H}\in\Gamma^{(k)}(G)}\prod_{\gamma\in \mathcal{H}}\ell(\gamma)=\sum_{J\in\binom{[t]}{k}}\prod_{i\in J}\ell_i. \end{equation}
If $\ell_1+\ell_2+\cdots+\ell_t$ equals to a fixed value $m$, the right-hand side of~\eqref{eq:aux} is maximized when the tuple $(\ell_1,\ell_2,\ldots,\ell_t)$ is fair as defined below. This maximality result is stated in Lemma~\ref{lemma:fairness}.

\begin{definition}[\cite{2023-RomeroSafe}]
A tuple $(x_1,x_2,\ldots,x_t)\in \mathbb{Z}_{+}^t$ is \emph{fair} if $|x_i-x_j|\leq 1$ for all $i,j \in \{1,2,\ldots,t\}$. A graph $G$ is \emph{fair} if it is a $2$-connected simple graph having more edges than vertices such that the tuple whose entries are the lengths of all the chains in $G$ is fair. 
\end{definition}

\begin{lemma}[\cite{2023-RomeroSafe}]\label{lemma:fairness} Let $k$ and $t$ be integers such that $2 \leq k\leq t$ and let
\[ \phi^{(k)}_{t}(\ell_1,\ell_2,\ldots,\ell_t)=\sum_{J\in\binom{[t]}{k}}\prod_{i\in J}\ell_i\qquad\text{for each }(\ell_1,\ell_2,\ldots,\ell_t)\in\mathbb Z_+^t. \]
Let $m$ be any integer such that $m \geq t$ and let $L_{t,m}=\{(\ell_1,\ell_2,\ldots,\ell_t)\in\mathbb Z_+^t:\ell_1+\ell_2+\cdots+\ell_t=m\}$. The following assertions hold.
\begin{enumerate}[label=(\roman*)]
\item The maximum of $\phi_t^{(k)}(\ell_1,\ell_2,\ldots,\ell_t)$ as $(\ell_1,\ell_2,\ldots,\ell_t)$ ranges over $L_{t,m}$ is attained precisely at those tuples that are fair. 
\item  Let $\Phi^{(k)}_{t}(m)$ be the maximum attained by $\phi_t^{(k)}(\ell_1,\ell_2,\ldots,\ell_t)$ in $L_{t,m}$. 

If $2\leq k<t$ then $\Phi^{(k)}_{t}(m)<\Phi^{(k)}_{t+1}(m)$.
\end{enumerate}
\end{lemma}


\begin{definition}[\cite{2023-RomeroSafe}]\label{def:induced}
Let $G$ be a $2$-connected simple graph on more edges than vertices. Let $\{f_1,f_2,\ldots,f_k\}$ be a $k$-edge-cut of $D(G)$. For each $i \in \{1,2,\ldots,k\}$, let $\gamma_i$ be the chain of $G$ corresponding to the edge $f_i$ of $D(G)$. We say a $k$-edge-cut $\{e_1,e_2,\ldots,e_k\}$ of $G$ \emph{is induced by $\{f_1,f_2,\ldots,f_k\}$} if $e_i \in \gamma_i$ for each $i \in \{1,2,\ldots,k\}$. Moreover,
\begin{enumerate}[label=(\roman*)]
\item if $\{f_1,f_2,\ldots,f_k\}$ is vertex-separating, then the $k$-edge-cut $\{e_1,e_2,\ldots,e_k\}$ is called \emph{Type-V};
\item if $\{f_1,f_2,\ldots,f_k\}$ is edge-separating but not vertex-separating, then the $k$-edge-cut  
$\{e_1,e_2,\ldots,e_k\}$ is called \emph{Type-E};
\item if $\{f_1,f_2,\ldots,f_k\}$ is nontrivial, then the $k$-edge-cut $\{e_1,e_2,\ldots,e_k\}$ is called \emph{Type-N}.
\end{enumerate}
The number of Type-V, Type-E, and Type-N $k$-edges-cuts of $G$ is denoted $\mu_k^{\mathrm V}(G)$, $\mu_k^{\mathrm E}(G)$, and $\mu_k^{\mathrm N}(G)$, respectively. The total number of induced $k$-edge-cuts of $G$ is denoted $\mu_k^{\mathrm I}(G)$; i.e.,
\begin{equation*}
   \mu_k^{\mathrm I}(G)=\mu_k^{\mathrm V}(G)+\mu_k^{\mathrm E}(G)+\mu_k^{\mathrm N}(G).
\end{equation*}
\end{definition}

Notice that, by the definition of $\mu_k^{\mathrm I}(G)$, it coincides with the third term of the right-hand side of~\eqref{eq:mk}, i.e., 
\begin{equation}\label{eq:induced}
\mu_k^{\mathrm I}(G)=\sum_{\mathcal{H}\in\Gamma^{(k)}_-(G)}\prod_{\gamma\in \mathcal{H}}\ell(\gamma).   
\end{equation}
This fact combined with Lemma~\ref{lemma:fairness} leads to the following result.

\begin{proposition}[\cite{2023-RomeroSafe}]\label{prop:mu^I} Let $k$ and $t$ be positive integers such that $2\leq k\leq t$ and let $\mathcal S$ be a nonempty set of fair graphs on $n$ vertices and $m$ edges such that $m>n$ and having precisely $t$ chains. Then, as $G$ ranges over $\mathcal S$, the minimum of $\mu_k(G)$ is attained precisely in the same graphs $G$ where the minimum of $\mu_k^{\mathrm I}(G)$ is attained. \end{proposition}

\section{Minimization of $2$-edge-cuts and $3$-edge-cuts}\label{section:nodefair}
Let $p$ and $i$ be nonnegative integers such that $p\geq 2$. The purpose of this section is two-fold. First, we characterize the set $\mathcal{C}_{2p+i,3p+i}^{2}$ consisting of all min-$\mu_2$ graphs in $\mathcal{C}_{2p+i,3p+i}$. 
Then, we find the set $\mathcal{C}_{2p+i,3p+i}^{3}$ consisting of all graphs with minimum number of $3$-edge-cuts among all graphs in $\mathcal{C}_{2p+i,3p+i}^{2}$. 

The following lemma can be proved using handshaking.
\begin{lemma}\label{lemma:hand}
Let $p$ and $i$ be nonnegative integers such that $p\geq 2$. If $G$ is a 2-connected graph in 
$\mathcal{C}_{2p+i,3p+i}$ then all the following statements hold,
\begin{enumerate}[label=(\roman*)]
\item\label{item1} Its distillation $D(G)$ has at most $3p$ edges and $\delta(D(G))\geq 3$. 
\item\label{item2} If $D(G)$ has precisely $3p$ edges then $D(G)$ is cubic.
\item\label{item3} The graph $G$ has at most $3p$ chains. The equality holds if and only if $D(G)$ is cubic. 
\end{enumerate}
\end{lemma}    

Lemma~\ref{lemma:min2} follows just finding a simultaneous minimization of each of the $3$ terms that appear on the right hand side of Equation~\eqref{eq:mk}. The first term does not depend on the choice of the $2$-connected graph $G$. By Lemma~\ref{lemma:fairness}, the second term is minimum precisely when $G$ is a fair graph and has the greatest number of chains (i.e., $D(G)$ is a cubic graph in $\mathcal{C}_{2p,3p}$), while the third term is minimum precisely when the set $\Gamma^{(2)}_-(G)$ is empty, or equivalently, when $\lambda(D(G))=3$. 

\begin{lemma} \label{lemma:min2}
Let $p$ and $i$ be nonnegative integers such that $p\geq 2$. Among all 2-connected graphs $G$ in $\mathcal{C}_{2p+i,3p+i}$, the number $\mu_2(G)$ is minimum if and only if $G$ is a fair graph whose distillation $D(G)$ is a cubic graph in $\mathcal{C}_{2p,3p}$ such that $\lambda(D(G))=3$.
\end{lemma}

The following concept will be used to characterize the set  $\mathcal{C}_{2p+i,3p+i}^{3}$ for each pair of nonnegative integers $p$ and $i$ such that $p\geq 2$.
\begin{definition}
A fair graph is \emph{vertex-fair} if the sum-length of the chains that are incident to any fixed vertex is a fair tuple.    
\end{definition}

\begin{remark}\label{remark:Mobius}
M\"obius graph $M_p$ has only vertex-trivial $3$-edge cuts thus 
$\mu_3(M_p)=2p$ and $M_p$ is a cubic min-$\mu_3$ graph in  $\mathcal{C}_{2p,3p}$. For each nonnegative integer $i$ there exists some vertex-fair graph $G_{2p+i}$ in $\mathcal{C}_{2p+i,3p+i}$ whose distillation is $M_p$.
\end{remark}

The following technical lemma will give us all Type-V $3$-edge-cuts for each graph $G$ in  $\mathcal{C}_{2p+i,3p+i}^{2}$.
\begin{lemma}\label{lemma:technical}
Let $p$ and $i$ be nonnegative integers such that $p\geq 2$. 
Let $r$ and $s$ be the only integers such that 
$r\in \{0,1,\ldots,3p-1\}$ and $3p+i=3ps+r$. For each $j\in \{0,\ldots,3\}$, define $q_j$ such that $q_j(s)=(s+1)^js^{3-j}$. Denote $S_{p,r} = \{(x_1,\ldots,x_{2p}) \in \{0,\ldots,3\}^{2p}: x_1+\ldots+x_{2p}=2r\}$ and let

\[ g_s(x_1,\ldots,x_{2p})= \sum_{j=1}^{2p}q_{x_j}(s) \qquad\text{for each }(x_1,\ldots,x_{2p})\in S_{p,r}. \]

Then, the inequalities $q_3(s)>q_2(s)>q_1(s)>q_0(s)$ hold and the function $g_s(x_1,\ldots,x_{2p})$ attains its minimum in $S_{p,r}$ precisely at those tuples  $(x_1,\ldots,x_{2p})$ in $S_{p,r}$ which are fair. 
\end{lemma}

Let $G$ be any graph in $\mathcal{C}_{2p+i,3p+i}^{2}$ and let $v_1,v_2,\ldots,v_{2p}$ be the $2p$ vertices that are endpoints of some chain in $G$. If we call $x_i$ to the number of incident chains of $v_i$ whose lengths are $s+1$, then $G$ has precisely 
$g_s(x_1,x_2,\ldots,x_{2p})$ Type-V $3$-edge-cuts. If $G$ has the least number of Type-V $3$-edge-cuts among all graphs in $\mathcal{C}_{2p+i,3p+i}^{2}$ and no Type-E neither Type-N $3$-edge-cuts, then $G$ will have the least value of $\mu_3^{\mathrm{I}}$ in $\mathcal{C}_{2p+i,3p+i}^{2}$. Lemma~\ref{lemma:min3} then follows from Proposition~\ref{prop:mu^I}.

\begin{lemma}\label{lemma:min3}
Let $p$ and $i$ be nonnegative integers such that $p\geq 2$. The number $\mu_3(G)$ is minimum among all graphs $G$ in $\mathcal{C}_{2p+i,3p+i}^{2}$ if and only if $G$ is vertex-fair and $D(G)$ is a cubic graph in $\mathcal{C}_{2p,3p}$ having only vertex-trivial $3$-edge-cuts.
\end{lemma}

\section{Main Result}\label{section:main}
In this section we will prove that there are finitely many uniformly most reliable graphs of corank $5$. Let $n$ be any integer such that $n \geq 8$. First, we will find the only graph $G_n$ in $\mathcal{C}_{n,n+4}$ such that $\mathcal{C}_{n,n+4}^{4}=\{G_n\}$. Then we will construct, for each integer $n$ such that $n\geq 13$, a graph $H_n$ in $\mathcal{C}_{n,n+4}$. Finally, we will prove that there exists some integer $n_0$ such that $\mu_{5}(H_n)<\mu_5(G_n)$ for all $n\geq n_0$. The main theorem will then follow from Lemma~\ref{lemma:strategy}. 

Now, let us find $\mathcal{C}_{n,n+4}^{3}$. Choosing $p=4$ and $i=n-2p$ in Lemma~\ref{lemma:min3} gives that a graph $G$ belongs to  $\mathcal{C}_{n,n+4}^{3}$ if and only if $G$ is vertex-fair and 
$D(G)$ is a cubic graph in $\mathcal{C}_{8,12}$ having only vertex-trivial $3$-edge-cuts. Bussemake~\cite{Bussemake1977} shows that there are only $5$ nonisomorphic cubic graphs in $\mathcal{C}_{8,12}$. There are precisely $2$ such graphs having only vertex-trivial $3$-edge-cuts, namely, Wagner graph $W$ and the cube graph $Q$ which are depicted in Figure~\ref{figure:cubicosB}. 
We obtained the following result.
\begin{lemma}\label{lemma:QoW}
Let $n$ be an integer such that $n\geq 8$. A graph $G$ belongs to $\mathcal{C}_{n,n+4}^{3}$ if and only if $G$ is vertex-fair and $D(G) \in \{Q,W\}$.
\end{lemma}

The following concept will be useful to define the only graph $G_n$ such that $\mathcal{C}_{n,n+4}^4=\{G_n\}$ for each $n\geq 8$. 
\begin{definition}
Let $H$ be any $2$-connected simple cubic graph and let $Y$ be a set of edges of $H$. We denote $H\odot_s Y$ the graph that arises from $H$ by subdividing $s$ times each edge in $Y$ and $s-1$ times each edge not in $Y$. A \emph{vertex-fair subdivision of $H$} is any vertex-fair graph isomorphic to $H\odot_s Y$ for some $Y\subseteq E(H)$. 
\end{definition}

\begin{figure}
\begin{center}
\scalebox{0.62}{
\begin{tabular}{c}
\begin{tikzpicture}[ scale=1.2, nodo/.style={circle,draw=black!120,fill=white!120,inner sep=0pt,minimum size=5mm}]

\node[nodo] (1) at (0,2) {$1$};
\node[nodo] (2) at (1.4142,1.4142) {$2$};
\node[nodo] (3) at (2,0) {$3$};
\node[nodo] (4) at (1.4142,-1.4142) {$4$};
\node[nodo] (5) at (0,-2) {$5$};
\node[nodo] (6) at (-1.4142,-1.4142) {$6$};
\node[nodo] (7) at (-2,0) {$7$};
\node[nodo] (8) at (-1.4142,1.4142) {$8$};

\draw (1) to (2);
\draw (2) to (3);
\draw (3) to (4);
\draw (4) to (5);
\draw (5) to (6);
\draw (6) to (7);
\draw (7) to (8);
\draw (8) to (1);
\draw (1) to (5);
\draw (2) to (6);
\draw (3) to (7);
\draw (4) to (8);
\end{tikzpicture} 
\\
$W$
\end{tabular}
\begin{tabular}{c}
\begin{tikzpicture}[ scale=1.2, nodo/.style={circle,draw=black!120,fill=white!120,inner sep=0pt,minimum size=5mm}]

\node[nodo] (1) at (0,2) {$1$};
\node[nodo] (2) at (1.4142,1.4142) {$2$};
\node[nodo] (3) at (2,0) {$3$};
\node[nodo] (4) at (1.4142,-1.4142) {$4$};
\node[nodo] (5) at (0,-2) {$5$};
\node[nodo] (6) at (-1.4142,-1.4142) {$6$};
\node[nodo] (7) at (-2,0) {$7$};
\node[nodo] (8) at (-1.4142,1.4142) {$8$};

\draw (1) to (2);
\draw (2) to (3);
\draw (3) to (4);
\draw (4) to (5);
\draw (5) to (6);
\draw (6) to (7);
\draw (7) to (8);
\draw (8) to (1);
\draw (1) to (6);
\draw (2) to (5);
\draw (3) to (8);
\draw (4) to (7);
\end{tikzpicture} %
\\
$Q$
\end{tabular}
}
\caption{Wagner graph $W$ and the cube $Q$. \label{figure:cubicosB}}
\end{center}
\end{figure}
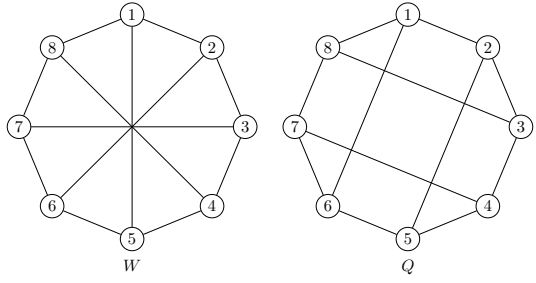

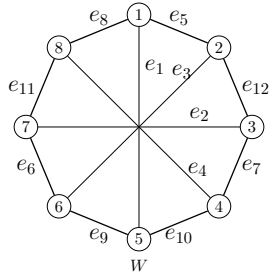
\begin{figure}
\begin{center}
\scalebox{0.62}{
\begin{tabular}{c}
\begin{tikzpicture}[ scale=1.2, nodo/.style={circle,draw=black!120,fill=white!120,inner sep=0pt,minimum size=5mm}]

\node[nodo] (1) at (0,2) {$1$};
\node[nodo] (2) at (1.4142,1.4142) {$2$};
\node[nodo] (3) at (2,0) {$3$};
\node[nodo] (4) at (1.4142,-1.4142) {$4$};
\node[nodo] (5) at (0,-2) {$5$};
\node[nodo] (6) at (-1.4142,-1.4142) {$6$};
\node[nodo] (7) at (-2,0) {$7$};
\node[nodo] (8) at (-1.4142,1.4142) {$8$};

\draw[right,pos=0.2] (1) edge node {\Large{$e_1$}}  (5); 
\draw[above,pos=0.2] (3) edge node {\Large{$e_2$}}  (7); 
\draw[above,pos=0.2] (2) edge node {\Large{$e_3$}}  (6); 
\draw[right,pos=0.2] (4) edge node {\Large{$e_4$}}  (8); 
\draw[thick,above] (1) edge node {\Large{$e_5$}}  (2); 
\draw[thick,left] (6) edge node {\Large{$e_6$}}  (7); 
\draw[thick,right] (3) edge node {\Large{$e_7$}}  (4); 
\draw[thick,above] (8) edge node {\Large{$e_8$}}  (1); 
\draw[thick,below] (5) edge node {\Large{$e_9$}}  (6); 
\draw[thick,below] (4) edge node {\Large{$e_{10}$}}  (5); 
\draw[thick,left] (7) edge node {\Large{$e_{11}$}}  (8); 
\draw[thick,right] (2) edge node {\Large{$e_{12}$}}  (3);
\end{tikzpicture}
\\
$W$
\end{tabular}
}
\end{center}
\caption{Edges $e_1,e_2,\ldots,e_{12}$ in the Wagner graph $W$. \label{figure:Wedges} 
} 
\end{figure}

\begin{definition}\label{definition:Gn}
Define, for each integer $n$ such that $n\geq 8$, 
the vertex-fair subdivision $G_n$ of $W$ as follows. 
Label the edges of $W$ as in Figure~\ref{figure:Wedges}. 
Let $r$ and $s$ be the unique integers such that $n+4=12s+r$ and $r\in \{0,\ldots,11\}$. 
If $r=0$, let $X_0=\emptyset$; if $r=8$, let $X_8=\{e_1,e_2,e_3,e_4,e_6,e_8,e_{10},e_{12}\}$; otherwise, $X_r=\{e_1,e_2,\ldots,e_r\}$.  
Then, $G_n$ is defined as $W\odot_s X_r$.     
\end{definition}

By Lemma~\ref{lemma:QoW}, each graph in $\mathcal{C}_{n,n+4}^3$ is either a vertex-fair subdivision of $Q$ or a vertex-fair subdivision of $W$. In order to determine $\mathcal{C}_{n,n+4}^4$ we will find closed forms for $\mu_4^{\mathrm V}(G)$, $\mu_4^{\mathrm E}(G)$, and $\mu_4^{\mathrm N}(G)$. The following definition will be considered, where the labels of the vertices and edges of $Q$ and $W$ are given in Figure~\ref{figure:cubicosB}. 

\begin{definition} 
Let $G$ be any graph in $\mathcal{C}_{n,n+4}^3$ such that $G=D(G)\odot_s X$. 
\begin{enumerate}
    \item[\textbullet] For each $i\in \{0,1,2,3\}$, we will denote $p_i(G)$ the number of vertices in $D(G)$ that are incident to precisely $i$ edges of $X$.
    \item[\textbullet] For each $j\in \{0,1,2,3,4\}$, we will denote $q_j(G)$ the number of edges in $D(G)$ that are incident to precisely $j$ edges of $X$.
    \item[\textbullet] If $k\in \{1,2,3\}$ and $D(G)=Q$, then $z_k(G)=|F_k\cap X|$, where $F_1 = \{23,45,67,81\}$, $F_2 = \{12,38,47,56\}$, and $F_3=\{16,25,34,78\}$. 
    \item[\textbullet] If $k\in \{1,2,3\}$ and $D(G)=W$, then $z_k(G)=|E_j\cap X|$, where 
$E_1= \{12,34,56,78\}$, $E_2= \{23,45,67,81\}$, and $E_3=\{15,26,37,48\}$. 
\end{enumerate}
\end{definition}

Using the definition of Type-V, Type-E and Type-N edge-cuts and the fact that each graph $G$ in $\mathcal{C}_{n,n+4}^3$ is either a vertex-fair subdivision of 
$Q$ or a vertex-fair subdivision of $W$, the following lemma can be derived.
\begin{lemma}\label{lemma:induced-4-cuts}
Let $n$ be an integer such that $n\geq 8$ and let $r$ and $s$ be the unique integers 
such that $r\in \{0,1,\ldots,11\}$ and $n+4=12s+r$. If $G \in \mathcal{C}_{n,n+4}^3$, then 
\begin{align}
\mu_4^{\mathrm V}(G) &= \sum_{i=0}^{3}p_i(G)(s+1)^{i}s^{3-i}(n+4-3s-i), \label{muV}\\
\mu_4^{\mathrm E}(G) &= \sum_{i=0}^{4}q_i(G)(s+1)^{i}s^{4-i}, \label{muE}\\
\mu_4^{\mathrm N}(G) &= \sum_{i=1}^{3}(s+1)^{z_i(G)}s^{4-z_i(G)}, \text{ when } D(G)=Q, \label{muN}\\
\mu_4^{\mathrm N}(G) &= \sum_{i=1}^{2}(s+1)^{z_i(G)}s^{4-z_i(G)}, \text{ when } D(G)=W. \label{muN2}
\end{align}
\end{lemma}

Consider, for each integer $n$ such that $n\geq 8$, the only integers 
$r$ and $s$ given in the statement of Lemma~\ref{lemma:induced-4-cuts}. As each graph in $\mathcal{C}_{n,n+4}^{3}$ is either a vertex-fair subdivision of $Q$ or a vertex-fair subdivision of $W$, we can obtain for each $r$ and $s$ the list of all nonisomorphic graphs in $\mathcal{C}_{n,n+4}^3$. Then, we can find 
the value $\mu_4^{\mathrm{I}}(G)$ given by the sum $\mu_4^{\mathrm{V}}(G)+\mu_4^{\mathrm{E}}(G)+\mu_4^{\mathrm{N}}(G)$ to obtain the graph $G$ in $\mathcal{C}_{n,n+4}^{3}$
which minimizes $\mu_4^{\mathrm{I}}(G)$. As all graphs in $\mathcal{C}_{n,n+4}^{3}$ are fair, we can finally apply Proposition~\ref{prop:mu^I} to obtain the graph $G$ in 
$\mathcal{C}_{n,n+4}^3$ which minimizes $\mu_4(G)$. Using the previous methodology, the following result can be proved. 
 
\begin{proposition}\label{proposition:min4} 
For each positive integer $n$ such that $n\geq 8$, $\mathcal{C}_{n,n+4}^{4}=\{G_n\}$.
\end{proposition}

Now we will construct, for each integer $n$ such that $n\geq 13,$ a graph $H_{n}$ in $\mathcal{C}_{n,n+4}$. Finally, we will prove that there exists some positive integer $n_0$ such that $\mu_{5}(H_n)<\mu_5(G_n)$ for all $n\geq n_0$. 

\begin{definition}\label{definition:Hn}
Let $\ell_1,\ell_2,\ldots,\ell_{12}$ be positive integers. Denote $W(\ell_1,\ell_2,\ldots,\ell_{12})$ the graph that arises from $W$ by subdividing $\ell_i-1$ times the edge $e_i$ of $W$ for each $i\in \{1,2,\ldots,12\}$. 
For each $n\geq 13$, let $r$ and $s$ be the unique integers such that $r\in \{0,1,\ldots,11\}$ and 
$n+4=12s+r$. Define $\ell_i=s+1$ if $e_i \in X_r$, or $\ell_i=s$ otherwise. Observe  that 
$G_n$ is precisely $W(\ell_1,\ell_2,\ldots,\ell_{12})$. 
Define $H_{n}$ as $W(\ell_1',\ell_2',\ldots,\ell_{12}')$ where $\ell_{1}'=\ell_1+1$, $\ell_{5}'=\ell_{5}-1$, and $\ell_{i}'=\ell_{i}$ when $i\in \{1,2,\ldots,12\}-\{1,5\}$.
\end{definition}

Clearly, $H_n$ has as many vertices and edges as $G_n$. 
The definition of $H_n$ makes sense since 
we must have that $\ell_i'\geq 1$ for each $i\in \{1,2,\ldots,12\}$, and $\ell_5'\geq 1$ only when $n\geq 13$. 
Notice that $H_n$ is not a fair graph. Our goal is to prove that there exists some positive integer $n_0$ such that $\mu_{5}(H_n)<\mu_5(G_n)$ for all $n\geq n_0$. As both $G_n$ and $H_n$ are $2$-connected graphs on more edges than vertices, by Lemma~\ref{lemma:objetivo} and equations~\eqref{eq:aux} and~\eqref{eq:induced},
\begin{align}
  \mu_5(G_n)&=\binom{n+4}{5}-\phi_{12}^{(5)}(\ell_1,\ell_2,\ldots,\ell_{12})+\mu_5^{\mathrm I}(G_n), \label{W}\\
  \mu_5(H_n)&=\binom{n+4}{5}-\phi_{12}^{(5)}(\ell_1',\ell_2',\ldots,\ell_{12}')+\mu_5^{\mathrm I}(H_n).   \label{T}
\end{align}

In achieving our goal, we need to find the number of $5$-edge-cuts in $W$. 
\begin{lemma}[\cite{2023-RomeroSafe}]\label{lemma:count5}
Each nontrivial 5-edge-cut of $W$ is either $P_3$-separating or $C_4$-separating.
\end{lemma}

Let $G$ be either $G_n$ or $H_n$. Denote $\mu_5^{\mathrm P_3}(G)$ (resp. $\mu_5^{\mathrm C_4}(G)$) the number $5$-edge-cuts of $G$ induced by $P_3$-separating (resp. $C_4$-separating) edge-cuts of $D(G)$. By Lemma~\ref{lemma:count5},
\begin{equation}\label{equation:induced5}
\mu_5^{\mathrm I}(G)=\mu_5^{\mathrm V}(G)+\mu_5^{\mathrm E}(G)+\mu_5^{\mathrm P_3}(G)+\mu_5^{\mathrm C_4}(G).
\end{equation}

Let us find $\mu_5(G_n)-\mu_5(H_n)$ by replacing~\eqref{equation:induced5} into~\eqref{W} and \eqref{T},
\begin{align}
\mu_5(G_n)-\mu_5(H_n)&=\phi_{12}^{(5)}(\ell_1',\ell_2',\ldots,\ell_{12}')-\phi_{12}^{(5)}(\ell_1,\ell_2,\ldots,\ell_{12}) \notag \\
&+ (\mu_5^{\mathrm V}(G_n)-\mu_5^{\mathrm V}(H_n))+ (\mu_5^{\mathrm E}(G_n)-\mu_5^{\mathrm E}(H_n)) \notag \\ 
&+ (\mu_5^{\mathrm P_3}(G_n)-\mu_5^{\mathrm P_3}(H_n)) + (\mu_5^{\mathrm C_4}(G_n)-\mu_5^{\mathrm C_4}(H_n)) \label{eq:total}
\end{align}
The following $5$ lemmas give closed forms for each of the $5$ terms that appear on the right-hand side of equation~\eqref{eq:total}. The proof 
of Lemma~\ref{lemma:unfair} is straightforward. The rationale behind the proofs of Lemmas~\ref{lemma:vertextrivial}, \ref{lemma:edgetrivial}, \ref{lemma:nontrivial1} and \ref{lemma:nontrivial2} is similar. 
First, we identify which $5$-edge-cuts in $W$ we must consider. Then, we count the respective number of induced $5$-edge-cuts for both $G_n$ and $H_n$. 
Finally, we take the difference $\mu_5^{\mathrm X}(G_n)-\mu_5^{\mathrm{X}}(G)$ where $\mathrm{X}$ equals $\mathrm{V}$, $\mathrm{E}$, $\mathrm{P_3}$, or $\mathrm{C_4}$, respectively.  

\begin{lemma}\label{lemma:unfair}
For each integer $n$ such that $n\geq 13$, 
\begin{align}
\phi_{12}^{(5)}(\ell_1',\ell_2',\ldots,\ell_{12}')-\phi_{12}^{(5)}(\ell_1,\ell_2,\ldots,\ell_{12}) &= (\ell_5-\ell_1-1)\sum_{J \in \binom{[12]-\{1,5\}}{3}}\prod_{i\in J}\ell_i. \label{term:unfair}
\end{align}
\end{lemma}

\begin{lemma}\label{lemma:vertextrivial}
For each integer $n$ such that $n\geq 13$, 
\begin{align}
&\mu_5^{\mathrm V}(G_n)-\mu_5^{\mathrm V}(H_n)=
\ell_3\ell_6\ell_9\ell_{10}+ \ell_4\ell_7\ell_9\ell_{10}
-\ell_2\ell_3\ell_7\ell_{12}- \ell_3\ell_6\ell_9\ell_{12} \notag\\
&+(\ell_1+1-\ell_5)(\ell_2\ell_7\ell_{12}+\ell_4\ell_7\ell_{10}+\ell_3\ell_6\ell_9+\ell_2\ell_6\ell_{11}+\ell_4\ell_8\ell_{11}) \notag\\
&- (\ell_1+1-\ell_5)(\ell_8\ell_9\ell_{10}+\ell_3\ell_8\ell_{12}+\ell_4\ell_8\ell_{11}) \notag\\
&+\ell_3\ell_{12}(\sum_{\{i,j\}\in \binom{[12]-\{1,3,5,12\}}{2}}\ell_i\ell_j + (\ell_1+1-\ell_5)\sum_{i\in [12]-\{1,3,5,12\}}\ell_i) \notag\\
&+\ell_9\ell_{10} ((\ell_1+1-\ell_5)\sum_{i\in [12]-\{1,5,9,10\}}\ell_i - \sum_{\{i,j\}\in \binom{[12]-\{1,5,9,10\}}{2}}\ell_i\ell_j) \notag\\
&+\ell_8(\ell_1+1-\ell_5)\sum_{\{i,j\} \in \binom{[12]-\{1,5,8\}}{2}}\ell_i \ell_j. \label{term:vertextrivial}
\end{align}
\end{lemma}

\begin{lemma}\label{lemma:edgetrivial}
For each integer $n$ such that $n\geq 13$, 
\begin{align}
\mu_5^{\mathrm E}(G_n)-\mu_5^{\mathrm E}(H_n)&= 
\ell_8\ell_9\ell_{10}(n+4-\ell_1-\ell_5-\ell_8-\ell_9-\ell_{10}) \notag\\
&+\ell_6\ell_9\ell_{12}(n+5-\ell_3-2\ell_5 - \ell_6-\ell_9-\ell_{12}) \notag\\
&+\ell_2\ell_3\ell_{7}(n+5-\ell_2-\ell_3-2\ell_5-\ell_{7}-\ell_{12}) \notag\\
&- \ell_3\ell_8\ell_{12}(n+4-\ell_1-\ell_3-\ell_5-\ell_8-\ell_{12}) \notag \\
&- \ell_3\ell_6\ell_{10}(n+3-2\ell_1-\ell_3-\ell_6-\ell_9-\ell_{10}) \notag\\
&-\ell_4\ell_7\ell_{9}(n+3-2\ell_1-\ell_4-\ell_7-\ell_{9}-\ell_{10}) \notag \\ 
&+ \ell_4\ell_{11}(\ell_1+1-\ell_5)(n+4-\ell_1-\ell_4-\ell_5-\ell_8-\ell_{11}). \label{term:edgetrivial}
\end{align}
\end{lemma}

\begin{lemma}\label{lemma:nontrivial1}
For each integer $n$ such that $n\geq 13$, 
\begin{align}
\mu_5^{\mathrm P_3}(G_n)-\mu_5^{\mathrm P_3}(H_n)&=
(\ell_1+1-\ell_5)(
\ell_2\ell_4\ell_6+ \ell_6\ell_{10}\ell_{12}+ \ell_7\ell_{10}\ell_{11}) \notag\\
&+\ell_2\ell_3\ell_4\ell_{10}+\ell_3\ell_6\ell_7\ell_{11}+\ell_2\ell_6\ell_7\ell_9+\ell_4\ell_7\ell_8\ell_9 \notag\\
&+\ell_3\ell_6\ell_8\ell_{10}+\ell_2\ell_9\ell_{11}\ell_{12}+\ell_4\ell_9\ell_{10}\ell_{11} \notag\\
&-\ell_2\ell_3\ell_7\ell_8-\ell_2\ell_4\ell_9\ell_{12}-\ell_3\ell_4\ell_6\ell_{7}-\ell_2\ell_3\ell_{10}\ell_{11} \notag\\
&-\ell_3\ell_4\ell_{11}\ell_{12}-\ell_6\ell_8\ell_9\ell_{12}-\ell_7\ell_8\ell_9\ell_{11}. \label{term:nontrivial1}
\end{align}
\end{lemma}

\begin{lemma}\label{lemma:nontrivial2}
For each integer $n$ such that $n\geq 13$, 
\begin{align}
\mu_5^{\mathrm C_4}(G_n)-\mu_5^{\mathrm C_4}(H_n) &= \ell_7\ell_9\ell_{11}(n+5-2\ell_5-\ell_7-\ell_9-\ell_{11}). \label{term:nontrivial2}
\end{align}
\end{lemma}

\begin{proposition}\label{proposition:least}
There exists some positive integer $n_0$ such that, for all $n\geq n_0$, $\mu_5(H_n)< \mu_5(G_n)$. 
\end{proposition}
\begin{proof}
For each $n\geq 13$, let $r$ and $s$ be unique integers such that $r\in \{0,1,\ldots,11\}$ and $n+4=12s+r$. 
Consider, for each $j\in \{1,2,3,4,5\}$, the sequence $a^{(j)}_{n}$ defined as follows:
\begin{align*}
a^{(1)}_n &=\phi_{12}^{(5)}(\ell_1',\ell_2',\ldots,\ell_{12}')-\phi_{12}^{(5)}(\ell_1,\ell_2,\ldots,\ell_{12}),\\
a^{(2)}_n &=\mu_5^{\mathrm V}(G_n)-\mu_5^{\mathrm V}(H_n),\\
a^{(3)}_n &=\mu_5^{\mathrm E}(G_n)-\mu_5^{\mathrm E}(H_n),\\
a^{(4)}_n &=\mu_5^{\mathrm P_3}(G_n)-\mu_5^{\mathrm P_3}(H_n),\\
a^{(5)}_n &=\mu_5^{\mathrm C_4}(G_n)-\mu_5^{\mathrm C_4}(H_n).\\
\end{align*}
Observe that $\ell_1+1-\ell_5 \in \{1,2\}$, and $\ell_i\in \{s,s+1\}$ for all $i\in \{1,2,\ldots,12\}$. 
By Lemma~\ref{lemma:unfair}, $a^{(1)}_n \in O(n^3)$. Similarly, by Lemmas~\ref{lemma:vertextrivial}, \ref{lemma:edgetrivial} and \ref{lemma:nontrivial1}, $a^{(j)}_n \in O(n^3)$ for each $j \in \{2,3,4\}$. Additionally, by Lemma~\ref{lemma:nontrivial2} we can see that $a_n^{(5)}=7(\Bigl\lfloor \frac{n}{12} \Bigr\rfloor)^4+q_n$ for some sequence $q_n \in O(n^3)$. Consequently, the difference $\mu_5(G_n)-\mu_5(H_n)$, which equals $\sum_{j=1}^{5}a^{(j)}_n$, tends to infinity with respect to $n$ thus there exists some $n_0$ such that $\mu_5(G_n)-\mu_5(H_n)>0$ for all $n\geq n_0$, as required. \qed
\end{proof}

Replacing each of the expressions~\eqref{term:unfair}-\eqref{term:nontrivial2} into equation~\eqref{eq:total} gives that $\mu_5(G_n)-\mu_5(H_n)>0$ for each integer $n$ such that $n\geq 167$.\\

We are in position to prove the main result of this work. 
\begin{theorem}\label{theorem:main2}
There are finitely many UMRG of corank $5$.   
\end{theorem}
\begin{proof}
Let $n$ be any positive integer such that $n\geq n_0$, where $n_0$ is the least positive integer satisfying Proposition~\ref{proposition:least}. It is enough to show that there is no UMRG in $\mathcal{C}_{n,n+4}$. On the one hand, consider the graph $G_n$ given by Definition~\ref{definition:Gn}. 
By Proposition~\ref{proposition:min4}, we know that 
$\mathcal{C}_{n,n+4}^{4}=\{G_n\}$. On the other hand, consider the graph $H_n$ given by Definition~\ref{definition:Hn}. As $n\geq n_0$, by Proposition~\ref{proposition:least} we conclude  that $\mu_5(H_n)<\mu_5(G_n)$. The statement follows by Lemma~\ref{lemma:strategy}.    \qed
\end{proof}


\section*{Acknowledgments}
This work is partially supported by City University of New York project entitled \emph{On the problem of characterizing graphs with maximum number of spanning trees}, grant number 66165-00. The author wants to thank Dr. Mart\'in Safe for his helpful  comments that improved the presentation of this manuscript.

\bibliographystyle{plain}

\end{document}